\documentclass{article}
\usepackage[utf8]{inputenc}
\usepackage[english]{babel}

\usepackage[margin=1.25in]{geometry}
\usepackage{amssymb,amsfonts, amsmath, amsthm}
\usepackage[all,arc]{xy}
\usepackage{enumerate}
\usepackage{mathrsfs}
\usepackage{tikz}
\usepackage{subcaption}
\usetikzlibrary{positioning}

\newtheorem{theorem}{Theorem}[section]

\newtheorem{proposition}[theorem]{Proposition}
\newtheorem{lemma}[theorem]{Lemma}

\newtheorem{definition}[theorem]{Definition}
\newtheorem{definitions}[theorem]{Definitions}

\newtheorem{example}[theorem]{Example}

\newtheorem{remark}[theorem]{Remark}

\makeatletter
\let\c@equation\c@theorem
\makeatother

\usepackage[nottoc]{tocbibind}

\begin{document}

\title{Stable Cluster Variables}

\author{Grace Zhang}

\date{}

\maketitle

\begin{abstract}
In \cite{eager-franco}, Eager and Franco introduce a change of basis transformation on the F-polynomials of Fomin and Zelevinsky \cite{fz4}, corresponding to rewriting them in the basis given by fractional brane charges rather than quiver gauge groups. This transformation seems to display a surprising stabilization property, apparently causing the first few terms of the polynomials at each step of the mutation sequence to coincide. Eager and Franco conjecture that this transformation will always cause the polynomials to converge to a formal power series as the number of mutations goes to infinity, at least for quivers possessing certain symmetries and along periodic mutation sequences respecting such symmetries. In this paper, we verify this convergence in the case of the Kronecker and Conifold quivers. We also investigate convergence in the $F_0$ quiver, though the results here are still incomplete. We provide a combinatorial interpretation for the stable cluster variables in each appropriate case.
\end{abstract}

\maketitle

\tableofcontents

\setlength{\parskip}{1em}

\section{Introduction}

Cluster algebras were originally developed by Fomin and Zelevinsky \cite{fz1} in order to study total positivity and canonical bases in Lie theory. Since then, numerous connections have been discovered between cluster algebras and various areas of mathematics and physics. In brief, a cluster algebra is a particular type of commutative ring, along with some additional combinatorial structure. In particular, a cluster algebra is given by a distinguished subset of $n$ elements (a cluster seed) along with mutation rules describing how to generate another subset of $n$ elements (a cluster). Each cluster can again be mutated into other clusters. The cluster algebra is constructed from the seed by repeatedly mutating it in all possible ways into all possible clusters. 

In this paper we will focus on a limited type of cluster algebra, known as \textit{skew-symmetric
cluster algebras of geometric type}. It is common to describe such a cluster algebra by drawing a finite directed graph, or a quiver, with $n$ labelled vertices. Then, the initial seed corresponds to the elements in the labelling, and the mutation rules are encoded by the configuration of edges. 

By fixing an infinite sequence of mutations, one can generate an infinite sequence of polynomials, one at each step. Subject to certain constraints on the seed and mutation rules, these polynomials are called $F$-polynomials. $F$-polynomials were introduced by Fomin and Zelevinsky in \cite{fz4}. 

In the study of quiver gauge theories, F-polynomials are well-suited for describing a phenomenon known as Seiberg duality \cite{franco-musiker,chuang-jafferis, aganagic}. In Section 9.5 of \cite{eager-franco} Eager and Franco apply a transformation to F-polynomials that expresses them in terms of a natural alternate basis, rather than in terms of the initial cluster seed variables. In the language of quiver gauge theory, their transformation corresponds to a change of variables from quiver gauge groups to fractional brane charges. Once rewritten in this new basis, the F-polynomials appear to become convergent expressions, approaching some formal power series as the number of mutations goes to infinity. In their paper, they illustrate this apparent property for the first few F-polynomials generated by the dP1 quiver. They conjecture that this stabilization property should hold for some larger class of quivers possessing certain symmetries and along periodic mutation sequences respecting such symmetries.

The purpose of the current paper is to investigate the stabilization property of transformed F-polynomials in 3 specific cases. We verify convergence for the Kronecker and the Conifold quivers, and discuss partial results for the $F_0$ quiver. 

\section{Background}

We review the relevant definitions and background concepts here from a combinatorial perspective, using the language of quivers. For more complete treatments, including more generalized cluster algebras, see \cite{intro, fz1, fz4, williams, gr18, derksen-weyman-zelevinsky}.

\subsection{Cluster Algebras (of Skew-Symmetric, Geometric Type, without Frozen Vertices)}
A \textbf{quiver} is a finite directed graph, possibly with multiple edges, but with no self-loops and no 2-cycles. We will work with quivers whose $n$ vertices are labelled, and where the labels come from a field of rational functions in $n$ variables, $F = \mathbb{C}(x_1, \ldots , x_n)$. Let $Q$ be a labelled quiver with vertices $v_1, v_2, \ldots , v_n$, and labels $\ell_1, \ell_2, \ldots , \ell_n$, respectively. $Q$ encodes a cluster algebra, a sub-algebra of the ambient field. To explain how, we require some additional notions. 

Let $B$ be the (signed) adjacency matrix of the quiver. $B_{ij} =  \text{ number of edges }v_i \to v_j$ (with this entry being negative if the edges point $v_j \to v_i$). Hence, $B$ is a skew-symmetric matrix. For any vertex $v_k$ define \textbf{quiver mutation at vertex k}, to be a new quiver $\mu_k(Q)$, with the same vertices $v_1, \ldots , v_n$ but with a new adjacency matrix $B'$, and new vertex labellings $\ell_i'$, as follows: 
$$B'_{ij} = \begin{cases}
-B_{ij} & i=k \text{ or } j = k\\
B_{ij} + B_{ik}B_{kj} & B_{ik} > 0 \text{ and } B_{kj} > 0\\
B_{ij} - B_{ik}B_{kj} & B_{ik} < 0 \text{ and } B_{kj} < 0\\
B_{ij} & \text{ otherwise }
\end{cases}$$

$$\ell_i' = \ell_i \text{ for all } i\neq k$$
$$\ell_k' \ell_k= \prod\limits_{j \ : \ B_{jk} >0} B_{jk} \ell_j + \prod\limits_{j \ : \ B_{kj} >0} B_{kj} \ell_j$$

Quiver mutation is equivalently described by the following algorithm. For an example of quiver mutation, see Figure \ref{fig:kroneckermutation}. 
\begin{enumerate}
    \item Update the label at $v_k$ to $\ell_k'=$ $$\dfrac{\prod\limits_{\text{incoming arrows } v_j \to v_k}\ell_j + \prod\limits_{\text{outgoing arrows } v_k \to v_j} \ell_j}{\ell_k}$$
    \item For every 2-path $v_i \to v_k \to v_j$, draw an arrow $v_i \to v_j$. 
    \item If any self-loops or 2-cycles were newly created, delete them.
    \item Reverse all arrows incident to $v_k$.
\end{enumerate}

For any quiver $\mu_{k_1} \mu_{k_2} \ldots \mu_{k_m}(Q)$ reached by a finite sequence of mutations from $Q$, the collection of labels of its vertices is known as a \textbf{cluster}. Any single label in a cluster is a \textbf{cluster variable}. The cluster associated to $Q$, the initial quiver before any mutations, is known as the \textbf{cluster seed}. Finally, the \textbf{cluster algebra} defined by $Q$ is the sub-algebra of $F$ generated by all possible cluster variables that can be reached from $Q$ by any finite sequence of mutations. 

\subsection{Frozen Vertices and F-polynomials}

To define F-polynomials, we will first introduce a modification of $Q$, called its \textbf{framed quiver}, and denoted $Q'$. The vertices of $Q'$ are $\{v_1, v_2, \ldots , v_n, v_1', v_2', \ldots , v_n'\}$. That is, we add a new vertex $v_i'$ for each existing vertex $v_i \in Q$. We will consider the new vertices $v_i'$ to be ``\textbf{frozen}," meaning that we will never mutate the quiver there. The edges of $Q'$ retain all edges from $Q$, with the addition of a new edge $v_i \to v_i'$ for each $i$. 

Next we fix a \textbf{mutation sequence} of vertices $\mu = (v_{i_1}, v_{i_2}, \ldots )$ which includes only non-``frozen" vertices. Finally, we specify the labelling of $Q’$ to be 1 at any non-``frozen" vertex, and $y_i$ at any ``frozen" vertex $v_i'$. Hence, we will be generating cluster variables in $\mathbb{C}(y_1, y_2, \ldots y_n)$. We will mutate $Q'$ iteratively according to the fixed mutation sequence, generating a new cluster at each step. Note that only one new cluster variable $F_j$ is actually generated at  each step $j$ in the mutation sequence, with all other cluster variables remaining unchanged. The sequence $\{F_j\}$ is the sequence of \textbf{F-polynomials} generated by $Q'$ with mutation sequence $\mu$. These functions are polynomials in the variables $y_i$ with integer coefficients. \cite{fz4,derksen-weyman-zelevinsky}

\subsection{Stable Cluster Variables}

We now summarize the apparently stabilizing transformation on $F$-polynomials introduced by Eager and Franco \cite{eager-franco}. For the remainder of the paper, we will abbreviate quiver vertices $\{v_1, v_2, \ldots, v_n, v_1', v_2',  \\ \ldots, v_n'\}$ as $\{1, 2, \ldots, n, 1', 2', \ldots, n'\}$. In addition, we will denote the initial framed quiver as $Q_0$, and the quiver generated after the $k$ steps of the mutation sequence as $Q_k$. The F-polynomial generated at the $k$th step of the mutation sequence will be denoted $F_k$. 

For each quiver $Q_k$ in the sequence, define a matrix $C_k$, whose $ij$-th entry is the number of arrows $i' \to j$ in $Q_k$ (with this being negative if the arrows point $j \to i'$). In other words, $C_k$ is the lower left $n \times n$ submatrix of the signed adjacency matrix. We will refer to these matrices $C_k$ as \textbf{C-matrices}. The inverse C-matrix will provide the stabilizing transformation we are interested in. 

\begin{definition}
Given some monomial $m = y_1^{a_1}\cdot y_2^{a_2}\cdot \ldots \cdot y_n^{a_n} \in \mathbb{C}[y_1, \ldots, y_n]$, its \textbf{$C_k$-matrix transformation} is $$\tilde{m} := y_1^{b_1}\cdot y_2^{b_2}\cdot \ldots \cdot y_n^{b_n} \text{, where } \vec{b} = C_k^{-1} \vec{a}$$
Writing the $F$-polynomial $F_k = \sum cm$ as a linear combination of monomials $m$, extend its \textbf{$C$-matrix transformation} by linearity: $$\tilde{F}_k := \sum_m c\tilde{m}$$ 
\end{definition}

\par In the remainder of the paper, we present some examples where $\tilde{F}_k$ converges to a formal power series as $k \to \infty$. In the context of these examples, we refer to the transformed F-polynomials $\tilde{F}_k$ as \textbf{stable cluster variables}. 

\section{Kronecker Quiver}
\begin{figure}[!htb]
    \centering
    \begin{tikzpicture}[round/.style={circle, draw, outer sep= 3},]
        \node[round] (0) [] {0};
        \node[round] (1) [right=of 0] {1};
        \node[round] (0') [above=of 0] {0'};
        \node[round] (1') [above=of 1] {1'};
        
        \draw[transform canvas={yshift=-1mm},->] (1) -- (0);
        \draw[transform canvas={yshift=1mm},->] (1) -- (0);
        \draw[->] (0) to (0');
        \draw[->] (1) to (1');
    \end{tikzpicture}
    \caption{Framed Kronecker Quiver}
    \label{fig:kronecker}
\end{figure}
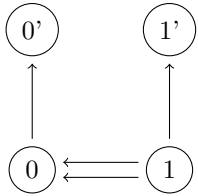

\par The first example we present is the Kronecker quiver, with its framed quiver pictured in Figure \ref{fig:kronecker}. We consider this example with respect to the mutation sequence $(0,1,0,1,\ldots)$. 

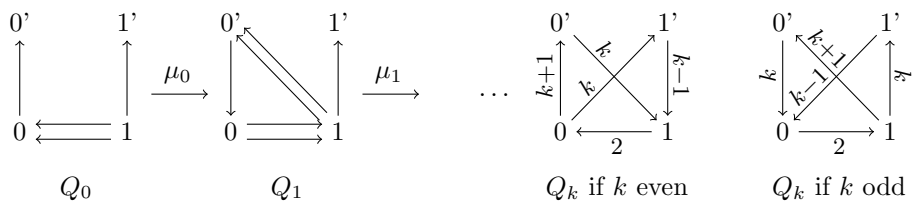
\begin{figure}[!htb]
\centering
\begin{tikzpicture}[node/.style={circle, inner sep = 0, outer sep= 2},]
        \node (title) at (0.75, -0.75) {$Q_0$};
        \node[node] (0) [] {0};
        \node[node] (1) [right=of 0] {1};
        \node[node] (0') [above=of 0] {0'};
        \node[node] (1') [above=of 1] {1'};
        
        \draw[transform canvas={yshift=-1mm},->] (1) -- (0);
        \draw[transform canvas={yshift=1mm},->] (1) -- (0);
        \draw[->] (0) to (0');
        \draw[->] (1) to (1');
        
        \node[color = black] () at (2.1,0.75) {$\mu_0$};
        \draw[->, color = black] (1.75,0.5) to (2.5,0.5);
\end{tikzpicture}
\begin{tikzpicture}[node/.style={circle, inner sep = 0, outer sep= 2},]
        \node (title) at (0.75, -0.75) {$Q_1$};
        \node[node] (0) [] {0};
        \node[node] (1) [right=of 0] {1};
        \node[node] (0') [above=of 0] {0'};
        \node[node] (1') [above=of 1] {1'};
        
        \draw[transform canvas={yshift=-1mm},->] (0) -- (1);
        \draw[transform canvas={yshift=1mm},->] (0) -- (1);
        \draw[transform canvas={xshift=-1mm},->] (1) -- (0');
        \draw[transform canvas={yshift=1mm},->] (1) -- (0');
        \draw[->] (0') to (0);
        \draw[->] (1) to (1');
        
        \node[color = black] () at (2.1,0.75) {$\mu_1$};
        \node () at (3.5, 0.5) {$\ldots$};
        \draw[->, color = black] (1.75,0.5) to (2.5,0.5);
\end{tikzpicture}
\begin{tikzpicture}[node/.style={circle, inner sep = 0, outer sep= 2},]
        \node (title) at (0.75, -0.75) {$Q_k$ if $k$ even};
        \node[node] (0) [] {0};
        \node[node] (1) [right=of 0] {1};
        \node[node] (0') [above=of 0] {0'};
        \node[node] (1') [above=of 1] {1'};
        
        \draw[->] (1) to (0); 
        \draw[->] (0) to (1'); 
        \draw[->] (0) to (0'); 
        \draw[->] (1') to (1); 
        \draw[->] (0') to (1); 
        
        \node[font = \small] () at (0.75, -0.2) {$2$};
        \node[rotate = 45, font = \small] () at (0.35, 0.6) {$k$};
        \node[rotate = 90, font = \small] () at (-0.2, 0.75) {$k$+1};
        \node[rotate = -90, font = \small] () at (1.6, 0.75) {$k-$1};
        \node[rotate = -45, font = \small] () at (0.6, 1.1) {$k$};
        
\end{tikzpicture}
\hspace{4.5mm}
\begin{tikzpicture}[node/.style={circle, inner sep = 0, outer sep= 2},]
        \node (title) at (0.75, -0.75) {$Q_k$ if $k$ odd};
        \node[node] (0) [] {0};
        \node[node] (1) [right=of 0] {1};
        \node[node] (0') [above=of 0] {0'};
        \node[node] (1') [above=of 1] {1'};
        
        \draw[->] (0) to (1); 
        \draw[->] (1') to (0); 
        \draw[->] (0') to (0); 
        \draw[->] (1) to (1'); 
        \draw[->] (1) to (0'); 
        
        \node[font = \small] () at (0.75, -0.2) {$2$};
        \node[rotate = 45, font = \small] () at (0.35, 0.6) {$k-$1};
        \node[rotate = 90, font = \small] () at (-0.2, 0.75) {$k$};
        \node[rotate = -90, font = \small] () at (1.6, 0.75) {$k$};
        \node[rotate = -45, font = \small] () at (0.6, 1.1) {$k$+1};
\end{tikzpicture}
\caption{The Kronecker quiver mutates with a predictable structure.}
\label{fig:kroneckermutation}
\end{figure}

\begin{figure}[!htb]
\centering
\begin{tabular}{|c|r|r|}
\hline
$k$ & $F_k$ & $\tilde{F}_k$\\
\hline
1 & $y_0 + 1$ & $\underline{y_0 + 1}$\\
2 & $y_0^2y_1 + y_0^2 + 2y_0 + 1$ & $y_0^2y_1^4 + \underline{2y_0y_1^2 + \underline{y_1 + 1}}
$\\
3 & $y_0^2y_1^2 + 2y_0^3y_1 + y_0^3 + 2y_0^2y_1 + 3y_0^2 + 3y_0 + 1$ & $y_0^9y_1^6 + 3y_0^6y_1^4 + 2y_0^5y_1^3 + \underline{3y_0^3y_1^2 + \underline{2y_0^2y_1 + \underline{y_0 + 1}}}$\\
4 & $\ldots + 6y_0^2y_1 + 4y_0^3 + 3y_0^2y_1 + 6y_0^2 + 4y_0 + 1$ & $\ldots + 3y_0^4y_1^6 + \underline{4y_0^3y_1^4 + \underline{3y_0^2y_1^3 + \underline{2y_0y_1^2 + \underline{y_1 + 1}}}}
$\\
\hline
\end{tabular}\\
    \caption{Table of the first few cluster variables, illustrating the stabilization property. The low order terms of the stable cluster variables match, up to a fluctuation between $y_0$ and $y_1$. (Entries in the last row are truncated).}
    \label{table:kronecker}
\end{figure}

Recall that we denote the sequence of quivers generated by the mutation sequence$\{Q_0, Q_1, \ldots\}$, the corresponding C-matrices $\{C_0, C_1, \ldots\}$, and the sequence of F-polynomials $\{F_1, F_2, \ldots\}$. We adopt the convention that $F_0 := 1$. 
\begin{lemma}
F-polynomials for the framed Kronecker quiver with mutation sequence $(0,1,0,1,\ldots)$ obey the following recurrence:
$$F_0 = 1$$ $$F_1 = y_0 + 1$$ $$F_kF_{k-2} = y_0^ky_1^{k-1} + F_{k-1}^2 \text{ for } k \geq 2$$
Further, the C-matrix and its inverse are given by

$C_k =
\begin{cases}
\begin{bmatrix}
-(k+1) &  k\\
-k & k-1
\end{bmatrix} & \text{if $k$ even}\\\\
\begin{bmatrix}
k & -(k+1) \\
k-1 & -k
\end{bmatrix} & \text{if $k$ odd}\\
\end{cases}$\hspace{4.5mm}
$C_k^{-1} =
\begin{cases}
\begin{bmatrix}
k-1 & -k \\
k & -(k+1)
\end{bmatrix} & \text{if $k$ even}\\\\
\begin{bmatrix}
k & -(k+1) \\
k-1 & -k
\end{bmatrix} & \text{if $k$ odd}\\
\end{cases}$
\end{lemma}
\begin{proof}
$Q_k$ follows the predictable structure shown in Figure \ref{fig:kroneckermutation}, which is easily verified by induction. From the definition of cluster mutation, the recurrence is immediately read off of the structure of $Q_k$. 
Similarly, the C-matrix is immediate, and its inverse easily computed. 
\end{proof}

The two possible forms of $C_k^{-1}$ differ from each other by permuting the rows. This discrepancy accounts for the fluctuation between variables $y_0$ and $y_1$ seen in Figure \ref{table:kronecker}. We will from now on remove this fluctuation by eliminating one case, in order to simplify computation. We choose, without loss of generality, to follow the case of odd $k$. 

\subsection{Row Pyramids} F-polynomials for certain types of quivers and mutation sequences have a known interpretation as the generating functions of a combinatorial object known as pyramid partitions \cite{young,eklp}.  Here we review this combinatorial interpretation in the case of the Kronecker quiver. 

\begin{definition}
\label{defn:rowpyramid}
Let the \textbf{row pyramid} of length $k$, $R_{k}$, be the two-layer arrangement of stones with $k$ white stones on the top layer and $k-1$ black stones on the bottom layer. The smallest three row pyramids are shown below. 
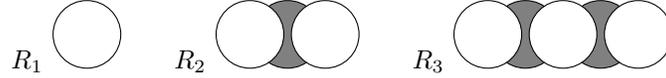
\begin{center}
$R_1$ \begin{tikzpicture}
        \foreach \x in {}
        \draw[fill = gray] (\x + 0.5, 0) circle (4.5mm);
        \foreach \x in {0}
        \draw[fill = white] (\x, 0) circle (4.5mm);
\end{tikzpicture}
\hspace{4.5mm}
$R_2$ \begin{tikzpicture}
        \foreach \x in {0}
        \draw[fill = gray] (\x + 0.5, 0) circle (4.5mm);
        \foreach \x in {0,1}
        \draw[fill = white] (\x, 0) circle (4.5mm);
\end{tikzpicture}
\hspace{4.5mm}
$R_3$ \begin{tikzpicture}
        \foreach \x in {0,1}
        \draw[fill = gray] (\x + 0.5, 0) circle (4.5mm);
        \foreach \x in {0,1,2}
        \draw[fill = white] (\x, 0) circle (4.5mm);
\end{tikzpicture}
\end{center}
\end{definition}

\begin{definitions}
\item A \textbf{partition} of $R_k$ is a stable configuration achieved by removing stones from $R_k$. By stable, we mean that if a stone is removed, then any stone lying on top of it must also be removed. (We will draw partitions by showing the non-removed stones).

\item For any partition $P$ of $R_k$, its \textbf{weight} is $$weight(P)= y_0^{\text{ \# white stones removed}}y_1^{\text{\# black stones removed}}$$
\end{definitions}

\begin{example}
A partition of $R_9$ with weight $y_0^5y_1$.
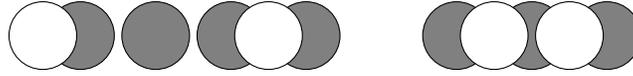
\begin{center}
\begin{tikzpicture}
        \foreach \x in {0,1,2,3,5,6,7}
        \draw[fill = gray] (\x + 0.5, 0) circle (4.5mm);
        \foreach \x in {0,3,6,7}
        \draw[fill = white] (\x, 0) circle (4.5mm);
\end{tikzpicture}
\end{center}
\end{example}

\begin{proposition}
$F_k$ is the following generating function, or partition function, for $R_{k}$.
$$F_k = \sum_{\text{Partitions $P$ of } R_{k}} \text{ weight }(P)$$
\end{proposition}
\begin{proof}
This may be proven inductively by verifying that the same recurrence on F-polynomials $F_k$ also holds for the generating function. 
\end{proof}

\subsection{Proof of Stabilization}
Next, we prove that the transformed F-polynomials in the case of the Kronecker quiver with the given mutation sequence do indeed stabilize in the limit to a formal power series. Afterwards, we give a combinatorial interpretation of that limit in terms of an infinitely long row pyramid. 

\begin{proposition}
$$\tilde{F}_k = \sum_{\text{Partitions of } R_{k}}\dfrac{(y_0^ky_1^{k-1})^{\text{ \# white stones removed}}}{(y_0^{k+1}y_1^{k})^{\text{\# black stones removed}}}$$
\end{proposition}
\begin{proof}
Since we know the precise form of the matrix $C_k$, we can verify that a monomial $m = y_0^{a} y_1^{b}$ transforms to $\tilde{m} = y_0^{k(a-b) - b} y_1^{k(a -b)-a}$. Then the proposition follows by transforming the generating function according to this rule, with $a = \#$ white stones removed, $b = \#$ black stones removed. Then, regroup terms. 
\end{proof}

\begin{definition}
A \textbf{simple partition} of $R_k$ is a partition of $R_k$ such that the removed white stones form one consecutive block, and no exposed black stones remain. The trivial partition with no stones removed is a simple partition.
\label{defn:simplepartitionrowpyramid}
\end{definition}

\begin{example}
A simple partition of $R_9$, with 3 white and 2 black stones removed.
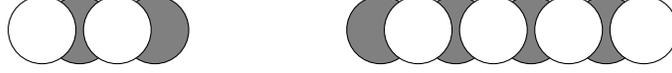
\begin{center}
\begin{tikzpicture}
        \foreach \x in {0,1,4,5,6,7}
        \draw[fill = gray] (\x + 0.5, 0) circle (4.5mm);
        \foreach \x in {0,1,5,6,7,8}
        \draw[fill = white] (\x, 0) circle (4.5mm);
\end{tikzpicture}
\end{center}
\end{example}

The idea of the proof of the stabilization property is that the stable terms in $\tilde{F}_k$ are the contributions exactly from the simple partitions. 

\begin{theorem}
For the Kronecker quiver with $\mu = (0,1,0,1, \ldots)$ $$ \lim_{k \to \infty} \tilde{F}_k = 1 + y_0 + 2y_0^2y_1 + 3y_0^3y_1^2 + 4y_0^4y_1^3 + \ldots$$ 
$$= 1 + \sum_{i=1}^\infty i \cdot y_0^{i} \cdot y_1^{i-1} 
$$
\label{thm:kronecker}
\end{theorem}
\begin{proof}
The term $1$ clearly stabilizes, since every $F_k$ includes $1$ as a term, coming from the trivial partition with no stones removed. The term remains unchanged under the linear transformation $C_k$. 

\underline{Claim}: For any monomial $y_0^ay_1^b \neq 1$ in $F_k$, we must have $a > b$.\\
In $R_k$ it is impossible to remove as many black stones as white stones, since a black stone can only be removed after both white stones on top of it have been removed. Since $F_k$ is the partition function for $R_{k}$ the claim follows.

\underline{Claim}: For any monomial $y_0^{a'}y_1^{b'} \neq 1$ in $\tilde{F}_k$, we must have $a' > b'$.\\ 
Let $m = y_0^ay_1^b$ be any monomial in $F_k$. $C_k$ transforms it to $\tilde{m} = y_0^{k(a-b)-b}y_1^{k(a-b)-a}$. By the previous claim, $b < a$. The claim follows.

So let $\tilde{m} = y_0^{a}y_1^{a-j}$ be a monomial, with $j \geq 1$.

\underline{Case 1}: $j = 1$. So $\tilde{m} = y_0^{a}y_1^{a-1}$. We claim that there is some $K$ such that for all $k \geq K$, $\tilde{m}$ appears in $\tilde{F}_k$ with coefficient $a$. 

Using the matrix $C_k$, $\tilde{m}$ appears in $\tilde{F_k}$ if and only if the term $y_0^{k-a+1}y_1^{k-a}$ appears in $F_k$. This term corresponds to a partition with $(k-a + 1)$ white stones removed and $(k-a)$ black stones removed. Note that since the difference is 1, it must be a simple partition. It is a straightforward combinatorial observation that there are $a$ such partitions whenever $k \geq a$ and 0 such partitions whenever $k < a$. So the claim holds with $K = a$.

\underline{Case 2}: $j\geq 2$. We claim that for sufficiently large $k$, $\tilde{m} = y_0^{a}y_1^{a-j}$ does not appear in $\tilde{F}_k$.

Suppose $\tilde{m}$ appears in $\tilde{F}_z$ for some $z$. Using the matrix $C_z$, it corresponds to the term $y_0^{zj - a + j}y_1^{zj-a}$ in $F_z$. If $\tilde{m}$ appears in $\tilde{F}_{z+1}$, then it corresponds to the term $y_0^{zj - a + 2j}y_1^{zj - a + j}$ in $F_{z+1}$, using $C_{z+1}$. That is, we add $j$ to each exponent. However, increasing from $z$ to $z+1$ adds only one stone of each color to $R_z$. So if $j\geq 2$, then after a finite number of steps, the exponents will grow too large for any possible partition. 
\end{proof}

\subsection{Combinatorial Interpretation of the Limit}

\par We now give a combinatorial interpretation for $\lim_{k \to \infty} \tilde{F}_k$. This interpretation will be generalized in the next section.

\begin{definition}
Let $R_\infty$ be the row pyramid extending infinitely toward the center as shown.

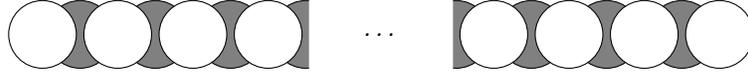
\begin{center}
\begin{tikzpicture}
        \foreach \x in {0,1,2,3}
        \draw[fill = gray] (\x + 0.5, 0) circle (4.5mm);
        \foreach \x in {0,1,2,3}
        \draw[fill = white] (\x, 0) circle (4.5mm);
        
        \foreach \x in {5,6,7,8}
        \draw[fill = gray] (\x + 0.5, 0) circle (4.5mm);
        \foreach \x in {6,7,8,9}
        \draw[fill = white] (\x, 0) circle (4.5mm);

        \draw[color = white, fill = white] (3.55,-0.5) rectangle (5.45, 0.5);
        \node () at (4.5,0) {\ldots};
\end{tikzpicture}
\end{center}
\end{definition}

\begin{definitions}
\item A \textbf{partition} of $R_\infty$ is a stable configuration achieved by removing \textbf{an infinite number of stones, such that only a finite number of stones remains}. 

\item A \textbf{simple partition} of $R_\infty$ is a partition of $R_\infty$ such that the removed white stones form one consecutive (infinite) block, and no exposed black stones remain. 

\item Define the \textbf{weight} of a partition $P$ of $R_\infty$ as $$weight(P) = y_0^{\text{\# non-removed white stones + 1}}y_1^{\text{\# non-removed black stones}}$$
Note that the number of non-removed white stones or black stones is actually the same. We have chosen to write the expression in this form in order to make it look more similar to a typical weight function. 
\end{definitions}

\begin{example} A simple partition of $R_\infty$ with weight $y_0^{4}y_1^{3}$. 
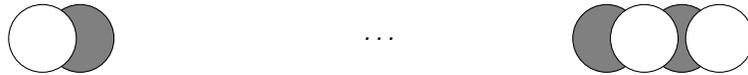
\begin{center}
\begin{tikzpicture}
        \foreach \x in {0}
        \draw[fill = gray] (\x + 0.5, 0) circle (4.5mm);
        \foreach \x in {0}
        \draw[fill = white] (\x, 0) circle (4.5mm);
        
        \foreach \x in {7,8}
        \draw[fill = gray] (\x + 0.5, 0) circle (4.5mm);
        \foreach \x in {8,9}
        \draw[fill = white] (\x, 0) circle (4.5mm);

        \draw[color = white, fill = white] (3.55,-0.5) rectangle (5.45, 0.5);
        \node () at (4.5,0) {\ldots};
\end{tikzpicture}
\end{center}
\end{example}

\begin{definition}
Define a partition function $$S = \sum_{\text{Simple partitions $P$ of } R_\infty} weight(P)$$
\label{defn:S}
\end{definition}

\begin{proposition}
$$\lim_{k \to \infty} \tilde{F}_k = 1 + S$$
\label{prop:S}
\end{proposition}

\begin{remark}
The constant term 1 appears unnatural here. In the next section, we will see how a generalization of $R_\infty$ gives rise to a analogous partition function. When viewed as a special case of this generalization, we will see that the constant term 1 actually arises naturally as a term in the partition function. 
\end{remark}

\section{Conifold Quiver}
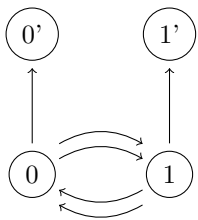
\begin{figure}[!htb]
    \centering
    \begin{tikzpicture}[round/.style={circle, draw, outer sep= 3},]
        \node[round] (0) [] {0};
        \node[round] (1) [right=of 0] {1};
        \node[round] (0') [above=of 0] {0'};
        \node[round] (1') [above=of 1] {1'};
        
        \draw[transform canvas={yshift=0mm},->] (0) to [bend left] (1);
        \draw[transform canvas={yshift=2mm},->] (0) to [bend left] (1);
        
        \draw[transform canvas={yshift=0mm},->] (1) to [bend left] (0);
        \draw[transform canvas={yshift=-2mm},->] (1) to [bend left] (0);
        
        \draw[->] (0) to (0');
        \draw[->] (1) to (1');
    \end{tikzpicture}
    \caption{Framed Conifold Quiver}
    \label{fig:conifold}
\end{figure}

\par The second example we present is the Conifold quiver, whose framed quiver is pictured in Figure \ref{fig:conifold}. We consider this example with respect to the mutation sequence $(0,1,0,1,\ldots)$. Note that the conifold is a quiver with 2-cycles, which according to the usual conventions we cannot mutate. For this example, we will mutate the quiver as usual, but at every step, remove any self-loops that were created. 

\par A table again suggests that the C-matrix transformation stabilizes the cluster variables. (Entries in the last two rows are truncated).

\begin{center}
\begin{tabular}{|c|r|r|}
\hline
$k$ & $F_k$ & $\tilde{F}_k$\\
\hline
1 & $y_0 + 1$ & $\underline{y_0 + 1}$\\
2 & $y_0^4y_1 + 2y_0^3y_1 + y_0^2y_1 + y_0^2 + 2y_0 + 1$ & $y_0^2y_1^5 + y_0^2y_1^4 + \underline{2y_0y_1^3 + 2y_0y_1^2 + \underline{y_1 + 1}}
$\\
3 & $\ldots +
6y_0^3y_1 + y_0^3 + 2y_0^2y_1 + 3y_0^2 + 3y_0 + 1$ & $\ldots + \underline{4y_0^4y_1^2 + 3y_0^3y_1^2 + \underline{2y_0^3y_1 + 2y_0^2y_1 + \underline{y_0 + 1}}}$\\
4 & $\ldots + 12y_0^3y_1 + 4y_0^3 + 3y_0^2y_1 +
6y_0^2 + 4y_0 + 1$ & $\ldots + \underline{4y_0^2y_1^4 + 3y_0^2y_1^3 + \underline{2y_0y_1^3 + 2y_0y_1^2 + \underline{y_1 + 1}}}$\\
\hline
\end{tabular}\\
\end{center}
\vspace{4.5mm}
\par Here is a larger number of stable terms. They do not seem to follow an obvious pattern:
\begin{multline*}
\ldots + 33y_0^{10}y_1^6 + 60y_0^9y_1^7 + 63y_0^9y_1^6 + 8y_0^8y_1^7 + 10y_0^9y_1^5 + 40y_0^8y_1^6 + 32y_0^8y_1^5\\ + 7y_0^7y_1^6 + 3y_0^8y_1^4 + 28y_0^7y_1^5 + 14y_0^7y_1^4 + 6y_0^6y_1^5 + 16y_0^6y_1^4 + 6y_0^6y_1^3 + 5y_0^5y_1^4\\ + 10y_0^5y_1^3 + y_0^5y_1^2 + 4y_0^4y_1^3 + 4y_0^4y_1^2 + 3y_0^3y_1^2 + 2y_0^3y_1 + 2y_0^2y_1 + y_0 + 1
\end{multline*}

The conifold also mutates with a predictable structure, and it is easy to see that the $C$-matrix has the same form as in Section 3 with the Kronecker quiver. As we did in Section 3, we will without loss of generality eliminate the even case in order to remove the fluctuation in variables in $\tilde{F}_k$:

\begin{lemma}
$C_k = C_k^{-1} = 
\begin{bmatrix}
k & -(k+1) \\
k-1 & -k
\end{bmatrix}$
\end{lemma}

\subsection{(2-color) Aztec Diamond Pyramids}

The F-polynomials generated by the conifold quiver also have a known combinatorial interpretation as the generating functions of certain pyramid partitions \cite{young,eklp}. We review this interpretation here. 

\begin{definition}
Let $AD_k$ be the 2-color Aztec diamond pyramid with $k$ white stones on the top layer. The first 4 cases are shown below. For example, in $AD_2$, there are a total of 4 white stones and 1 black stone. In $AD_3$, there are a total of 10 white stones and 4 black stones.  
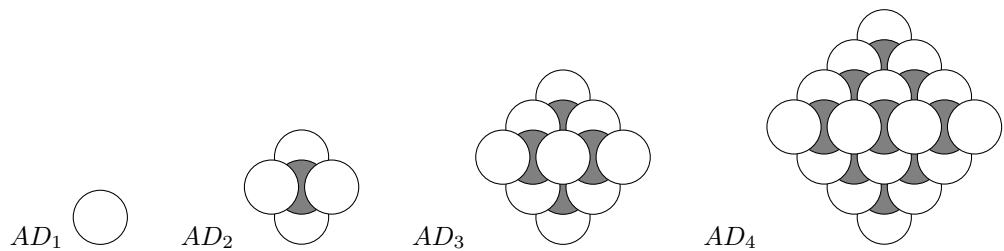
\begin{center}
$AD_1$ \begin{tikzpicture}[scale = 0.8]
        \foreach \x in {}
        \draw[fill = gray] (\x, 0) circle (4.5mm);
        \foreach \x in {0}
        \draw[fill = white] (\x, 0) circle (4.5mm);
\end{tikzpicture}
\hspace{4.5mm}
$AD_2$ \begin{tikzpicture}[scale = 0.8]
        \foreach \x in {0} \foreach \y in {-0.5, 0.5}
        \draw[fill = white] (\x, \y) circle (4.5mm);
        \foreach \x in {0} \foreach \y in {0}
        \draw[fill = gray] (\x, \y) circle (4.5mm);
        \foreach \x in {-0.5,0.5} \foreach \y in {0}
        \draw[fill = white] (\x, \y) circle (4.5mm);
\end{tikzpicture}
\hspace{4.5mm}
$AD_3$ \begin{tikzpicture}[scale = 0.8]
        \foreach \x in {0} \foreach \y in {-1, 0, 1}
        \draw[fill = white] (\x, \y) circle (4.5mm);
        
        \foreach \x in {0} \foreach \y in {-0.5, 0.5}
        \draw[fill = gray] (\x, \y) circle (4.5mm);
        
        \foreach \x in {-0.5,0.5} \foreach \y in {-0.5, 0.5}
        \draw[fill = white] (\x, \y) circle (4.5mm);
        
        \foreach \x in {-0.5, 0.5} \foreach \y in {0}
        \draw[fill = gray] (\x, \y) circle (4.5mm);
        \foreach \x in {-1,0,1} \foreach \y in {0}
        \draw[fill = white] (\x, \y) circle (4.5mm);
\end{tikzpicture}
\hspace{4.5mm}
$AD_4$ \begin{tikzpicture}[scale = 0.8]
        \foreach \x in {0} \foreach \y in {-1.5, -0.5, 0.5, 1.5}
        \draw[fill = white] (\x, \y) circle (4.5mm);
        
        \foreach \x in {0} \foreach \y in {-1, 0, 1}
        \draw[fill = gray] (\x, \y) circle (4.5mm);
        
        \foreach \x in {-0.5,0.5} \foreach \y in {-1, 0, 1}
        \draw[fill = white] (\x, \y) circle (4.5mm);
        
        \foreach \x in {-0.5, 0.5} \foreach \y in {-0.5, 0.5}
        \draw[fill = gray] (\x, \y) circle (4.5mm);
        
        \foreach \x in {-1,0,1} \foreach \y in {-0.5, 0.5}
        \draw[fill = white] (\x, \y) circle (4.5mm);
        
        \foreach \x in {-1, 0, 1} \foreach \y in {0}
        \draw[fill = gray] (\x, \y) circle (4.5mm);
        
        \foreach \x in {-1.5, -0.5, 0.5, 1.5} \foreach \y in {0}
        \draw[fill = white] (\x, \y) circle (4.5mm);
    \end{tikzpicture}

\end{center}

\end{definition}

\begin{definitions}
\item As in the previous section, a \textbf{partition} of $AD_k$ is a stable configuration achieved by removing stones from $AD_k$.

\item As in the previous section, for any partition $P$ of $AD_k$, its \textbf{weight} is $$weight(P)= y_0^{\text{ \# white stones removed}}y_1^{\text{\# black stones removed}}$$
\end{definitions}

\begin{example} A partition of $AD_4$ with weight $y_0^{4}y_1^{2}$
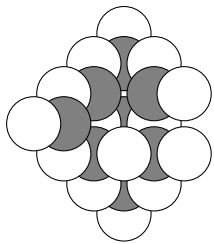
\begin{center}
\begin{tikzpicture}[scale = 0.8]
        \foreach \x in {0} \foreach \y in {-1.5, -0.5, 0.5, 1.5}
        \draw[fill = white] (\x, \y) circle (4.5mm);
        
        \foreach \x in {0} \foreach \y in {-1, 0, 1}
        \draw[fill = gray] (\x, \y) circle (4.5mm);
        
        \foreach \x in {-0.5,0.5} \foreach \y in {-1, 0, 1}
        \draw[fill = white] (\x, \y) circle (4.5mm);
        
        \foreach \x in {-0.5, 0.5} \foreach \y in {-0.5, 0.5}
        \draw[fill = gray] (\x, \y) circle (4.5mm);
        
        \foreach \x in {-1,1} \foreach \y in {-0.5, 0.5}
        \draw[fill = white] (\x, \y) circle (4.5mm);
        \draw[fill = white] (0, -0.5) circle (4.5mm);
        
        \foreach \x in {-1} \foreach \y in {0}
        \draw[fill = gray] (\x, \y) circle (4.5mm);
        
        \foreach \x in {-1.5} \foreach \y in {0}
        \draw[fill = white] (\x, \y) circle (4.5mm);
    \end{tikzpicture}
\end{center}
\label{example:A4partition}
\end{example}

\begin{theorem}[Elkies-Kuperberg-Larsen-Propp, 1992]
The F-polynomials are partition functions of $AD_k$.
$$F_k = \sum_{\text{Partitions $P$ of } AD_{k}} weight(P)$$
\end{theorem}
\begin{proof}
Proven in \cite{eklp}, using an interpretation by perfect matchings of graphs which is equivalent to our interpretation in terms of partitions of pyramids. 
\end{proof}

\subsection{Proof of Stabilization}

Note that each $AD_k$ can be decomposed into layers of row pyramids (Definition \ref{defn:rowpyramid}), such that the $j$th layer from the top contains $j$ row pyramids of length $k-j+1$. We will frequently refer to a row pyramid in the decomposition simply as a \textbf{row} of $AD_k$.

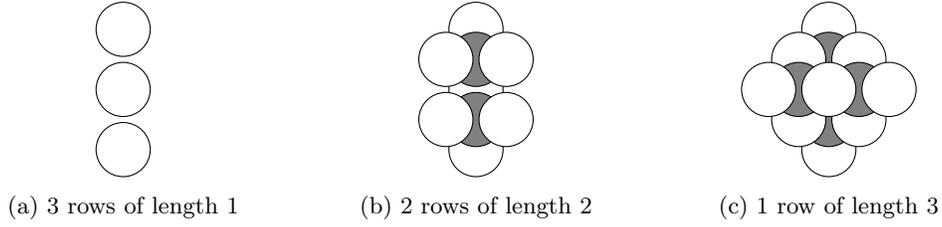
\begin{figure}[!htb]
    \centering
    
    \begin{subfigure}[b]{0.3\textwidth}
        \centering
        \begin{tikzpicture}[scale = 0.8]
        \foreach \x in {0} \foreach \y in {-1, 0, 1}
        \draw[fill = white] (\x, \y) circle (4.5mm);
    \end{tikzpicture}
        \caption{3 rows of length 1}
    \end{subfigure}
    \begin{subfigure}[b]{0.3\textwidth}
    \centering
        \begin{tikzpicture}[scale = 0.8]
        \foreach \x in {0} \foreach \y in {-1, 0, 1}
        \draw[fill = white] (\x, \y) circle (4.5mm);
        
        \foreach \x in {0} \foreach \y in {-0.5, 0.5}
        \draw[fill = gray] (\x, \y) circle (4.5mm);
        
        \foreach \x in {-0.5,0.5} \foreach \y in {-0.5, 0.5}
        \draw[fill = white] (\x, \y) circle (4.5mm);
    \end{tikzpicture}
        \caption{2 rows of length 2}
    \end{subfigure}
     \begin{subfigure}[b]{0.3\textwidth}
     \centering
        \begin{tikzpicture}[scale = 0.8]
        \foreach \x in {0} \foreach \y in {-1, 0, 1}
        \draw[fill = white] (\x, \y) circle (4.5mm);
        
        \foreach \x in {0} \foreach \y in {-0.5, 0.5}
        \draw[fill = gray] (\x, \y) circle (4.5mm);
        
        \foreach \x in {-0.5,0.5} \foreach \y in {-0.5, 0.5}
        \draw[fill = white] (\x, \y) circle (4.5mm);
        
        \foreach \x in {-0.5, 0.5} \foreach \y in {0}
        \draw[fill = gray] (\x, \y) circle (4.5mm);
        \foreach \x in {-1,0,1} \foreach \y in {0}
        \draw[fill = white] (\x, \y) circle (4.5mm);
    \end{tikzpicture}
        \caption{1 row of length 3}
    \end{subfigure}
    \caption{Row pyramid decomposition of $AD_3$, shown layer by layer.}
\end{figure}

A \textbf{simple partition} of $AD_k$ is a partition such that the restriction of the partition to each row pyramid is simple. (See Definition \ref{defn:simplepartitionrowpyramid}). For any partition $P$ of $AD_k$, for any row $r$ of $AD_k$, we call $r$ \textbf{altered} if at least one stone is removed from $r$.

Example \ref{example:A4partition} above is a simple partition with 2 altered rows. 

Analogous to the case of the Kronecker quiver in Section 3, the idea of the next proof is that the stable terms in $\tilde{F}_k$ are contributed by the simple partitions. 

\begin{theorem}
For the conifold, $\lim_{k\to \infty} \tilde{F}_k$ converges as a formal power series. 
\end{theorem}
\begin{proof}
The term 1 clearly stabilizes, since each $F_k$ includes 1 as a term, coming from the trivial partition with no stones removed. The linear transformation $C_k$ leaves the term unchanged. For the same reason as in the proof of Theorem \ref{thm:kronecker}, for every monomial $\tilde{m} = y_0^ay_1^b \neq 1$, if $\tilde{m}$ appears in $\tilde{F}_k$ for any $k$, then we must have $a > b$. 

\underline{Claim}: Let $\tilde{m} = y_0^ay_1^{a-j}$, with $j \geq 1$. For sufficiently large $k$, the terms in $F_k$ transforming to $\tilde{m}$ come only from simple partitions (possibly none). 

Suppose there is a $z$ such that there is a partition $P$ of $AD_z$ with weight $m_1$ transforming to $\tilde{m}$. Then it must be that $m_1 = y_0^{zj - a + j}y_1^{zj - a}$. In $P$, $j$ is the difference between the number of white stones and black stones removed. Let $S$ be the set of rows altered by $P$. Note that $P$ is a simple partition iff $|S| = j$, and $P$ is a non-simple partition iff $|S| < j$. 

Suppose $|S| < j$. 

If $F_{z+1}$ has a term transforming to $\tilde{m}$, it must be $m_2 = y_0^{zj - a + 2j}y_1^{zj - a + j}$. In other words, each exponent increases by $j$ from $m_1$. But increasing from $z$ to $z+1$ adds only one stone of each color to each row. So if $j > |S|$, then after a finite number of steps it will be impossible for any partition \textbf{altering exactly the rows in $\boldsymbol{S}$} to have a weight transforming to $\tilde{m}$. Since this is true for any set $S$ of fewer than $k$ rows, eventually the only possible partitions with weight transforming to $\tilde{m}$ will be simple partitions. 

This proof easily be can be modified to show that the following stronger claim holds: For sufficiently large $k$, the terms in $F_k$ transforming to $\tilde{m}$ come only from simple partitions, such that each altered row of the partition has more than $w$ stones removed, for any fixed $w$. 

\underline{Claim}: For sufficiently large $k$, the coefficient in front of $\tilde{m}$ in $\tilde{F}_k$ is constant. 

Assume $k$ is large enough that all partitions with weight transforming to $\tilde{m}$ are simple with each altered row having strictly greater than 1 stone removed, and that $k\geq j$. The second condition guarantees that $AD_k$ is large enough for every possible set of $j$ altered rows to exist. 

We construct a bijection $\phi$ between partitions of $AD_k$ with weight transforming to $\tilde{m}$ and partitions of $AD_{k+1}$ with weight transforming to $\tilde{m}$. 

Let $P$ be such a partition of $AD_k$. Since $P$ is simple, then for each altered row $r$, $P$ divides the unremoved stones in $r$ into two end sections, separated by the block of removed stones. (With each end section possibly empty). Increasing from $k$ to $k+1$ adds one stone of each color to each row. To get $\phi(P)$ we simply remove from each altered row one more stone of each color, such that the configuration of each end section is preserved. Since we removed $j$ additional stones of each color in total, $\phi(P)$ has weight transforming $\tilde{m}$, so the map is well-defined. (An example of $\phi$ is shown after the end of the proof). 

$\phi$ is bijective, since the inverse map is obvious (add one stone of each color such that the end section configurations are preserved) and well-defined. Well-definedness follows from the condition that each altered row has strictly greater than 1 stone removed, so the addition of stones still leaves the row an altered row. 
\end{proof}

\begin{example} 
$\phi(P)$ where $P$ is a simple partition of $AD_3$.
\begin{center}
$P$: \begin{tikzpicture}[scale = 0.8]
        \foreach \x in {0} \foreach \y in {-1, 0, 1}
        \draw[fill = white] (\x, \y) circle (4.5mm);
        
        \foreach \x in {0} \foreach \y in {-0.5, 0.5}
        \draw[fill = gray] (\x, \y) circle (4.5mm);
        
        \foreach \x in {-0.5} \foreach \y in {-0.5, 0.5}
        \draw[fill = white] (\x, \y) circle (4.5mm);
        \draw[fill = white] (0.5, -0.5) circle (4.5mm);
        
        \foreach \x in {-0.5} \foreach \y in {0}
        \draw[fill = gray] (\x, \y) circle (4.5mm);
        
        \foreach \x in {-1} \foreach \y in {0}
        \draw[fill = white] (\x, \y) circle (4.5mm);
    \end{tikzpicture}
    \hspace{5mm}
$\phi(P)$: \begin{tikzpicture}[scale = 0.8]
        \foreach \x in {0} \foreach \y in {-1.5, -0.5, 0.5, 1.5}
        \draw[fill = white] (\x, \y) circle (4.5mm);
        
        \foreach \x in {0} \foreach \y in {-1, 0, 1}
        \draw[fill = gray] (\x, \y) circle (4.5mm);
        
        \foreach \x in {-0.5,0.5} \foreach \y in {-1, 0, 1}
        \draw[fill = white] (\x, \y) circle (4.5mm);
        
        \foreach \x in {-0.5} \foreach \y in {-0.5, 0.5}
        \draw[fill = gray] (\x, \y) circle (4.5mm);
        \draw[fill = gray] (0.5, -0.5) circle (4.5mm);
        
        \foreach \x in {-1} \foreach \y in {-0.5, 0.5}
        \draw[fill = white] (\x, \y) circle (4.5mm);
        \draw[fill = white] (0, -0.5) circle (4.5mm);
        \draw[fill = white] (1, -0.5) circle (4.5mm);
        
        \foreach \x in {-1} \foreach \y in {0}
        \draw[fill = gray] (\x, \y) circle (4.5mm);
        
        \foreach \x in {-1.5} \foreach \y in {0}
        \draw[fill = white] (\x, \y) circle (4.5mm);
    \end{tikzpicture}
\end{center}
\end{example}

\subsection{Combinatorial Interpretation of the Limit}

\begin{definition}
Let $AD_\infty$ be the following Aztec diamond pyramid extending infinitely vertically and toward the middle:

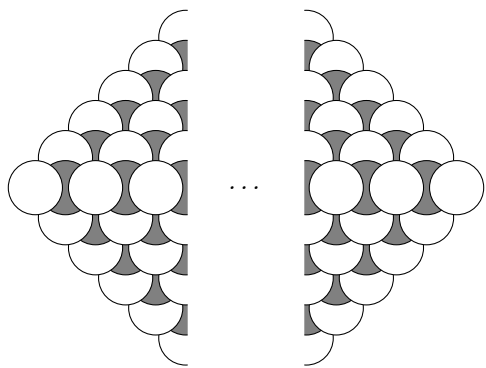
\begin{center}
\begin{tikzpicture}[scale = 0.8]
                
        \foreach \x in {-1,1}
            \foreach \y in {2.5, 1.5, 0.5, -0.5, -1.5, -2.5}
                \draw[fill = white] (\x, \y) circle (4.5mm);
                
        \foreach \x in {-1,1}
            \foreach \y in {2, 1, 0, -1,-2}
                \draw[fill = gray] (\x, \y) circle (4.5mm);
                
        \foreach \x in {-1.5, 1.5}
            \foreach \y in {2, 1, 0, -1, -2}
                \draw[fill = white] (\x, \y) circle (4.5mm);

        \foreach \x in {-1.5, 1.5}
            \foreach \y in {-1.5, -0.5, 0.5, 1.5}
                \draw[fill = gray] (\x, \y) circle (4.5mm);
                
        \foreach \x in {-2,-1,1,2}
            \foreach \y in {-1.5, -0.5, 0.5, 1.5}
                \draw[fill = white] (\x, \y) circle (4.5mm);
                
        \foreach \x in {-2, -1, 1, 2}
            \foreach \y in {-1, 0, 1}
                \draw[fill = gray] (\x, \y) circle (4.5mm);
                
        \foreach \x in {-2.5,-1.5,1.5,2.5}
            \foreach \y in {-1, 0, 1}
                \draw[fill = white] (\x, \y) circle (4.5mm);
        
        \foreach \x in {-2.5, -1.5, 1.5, 2.5}
            \foreach \y in {-0.5, 0.5}
                \draw[fill = gray] (\x, \y) circle (4.5mm);
                
        \foreach \x in {-3, -2,-1,1,2, 3}
            \foreach \y in {-0.5, 0.5}
                \draw[fill = white] (\x, \y) circle (4.5mm);
    
        \foreach \x in {-3, -2, -1, 1, 2, 3}
            \foreach \y in {0}
                \draw[fill = gray] (\x, \y) circle (4.5mm);
                
        \foreach \x in {-3.5, -2.5,-1.5,1.5,2.5, 3.5}
            \foreach \y in {0}
                \draw[fill = white] (\x, \y) circle (4.5mm); 
                
        \draw[color = white, fill = white] (-.96,-3) rectangle (.96, 3);
        \node () at (0,0) {\ldots};
                
\end{tikzpicture}
\end{center}

\end{definition}

\begin{definitions}
    \item A \textbf{partition} of $AD_\infty$ is a stable configuration achieved by removing stones from $AD_\infty$, such that for each row in its decomposition, \textbf{either no stones are removed, or an infinite number of stones are removed such that only a finite number of stones remains}. 
    
    \item A \textbf{simple partition} of $AD_\infty$ is a partition of $AD_\infty$ such that the restriction of the partition to each row is simple. 
    
    \item For any row $r$ of $AD_k$ or $AD_\infty$, define its \textbf{height} $h(r)$ as its distance from the top layer, such that the height of the top row is 0. Note that when $k$ is finite, $h(r) = k - $ (\# of white stones in $r$).
    
    \item For any partition $P$ of $AD_k$ or $AD_\infty$, define its \textbf{height} $$h(P) = \sum_{\text{altered rows $r$ of $P$}} h(r)$$
    
    \item For any partition $P$ of $AD_k$ or $AD_\infty$, let $$x(P) = \sum_{\text{altered rows $r$ of $P$}} (\text{\# non-removed white stones in $r$})$$
    Equivalently, $$x(P) = \sum_{\text{altered rows $r$ of $P$}} (\text{\# non-removed black stones in $r$})$$
\end{definitions}

\begin{definition}
    Define a partition function $$T = \sum_{P \text{ a simple partition of } AD_\infty} y_0^{x(P) + h(P) + \text{\# altered rows}}y_1^{x(P) + h(P)}$$
    \label{defn:T}
\end{definition}

\begin{proposition}
For the conifold $$\lim_{k \to \infty} \tilde{F}_k = T$$
\label{prop:T}
\end{proposition}

\begin{proof}
We show that if $P$ is a simple partition of $AD_k$ for finite $n$, and $m \neq 1$ is its weight, then $m$ transforms to $\tilde{m} = y_0^{x(P) + h(P) + \text{\# altered rows}}y_1^{x(P) + h(P)}$. 

Let $m = y_0^ay_1^{a-j}$. Then $m$ transforms to $\tilde{m} = y_0^{nj - a + j}y_1^{nj - a}$. Note that $j = \#$ altered rows. Also observe that $kj = h(P) +  \sum_{\text{altered rows $r$ of $P$}} \text{ length of $r$ }$. (Where length of $r = \#$ white stones in $r$ before any stones are removed). Hence $kj - a = h(P) + x(P)$. 
\end{proof}

\begin{remark}
Comparing Definition \ref{defn:T} to Definition \ref{defn:S} from the previous section reveals that the two are indeed analogous. In the case of the previous section, $h(P)$ is always 0, and the number of altered rows is always 1.
\end{remark}

\begin{remark}
Proposition \ref{prop:T} and Proposition \ref{prop:S} from the previous section appear not to be analogous, due to an additional "$+1$" constant term in the earlier result. However, this appearance is false. The "$+1$" term is hidden in the expression in Proposition \ref{prop:T}, coming from the trivial partition that removes no stones ).

However, after having seen the conifold quiver, there is now a natural way to revise Section 3 in order to embed the constant +1 term, by viewing diamond pyramids as a generalization of row pyramids, and then restricting back down from the generalized case. Replace all the definitions regarding partitions of $R_\infty$ with the definitions regarding partitions of $A_\infty$ (i.e. transport the definitions of altered rows, simple partitions, height, $x(P)$, and the partition function to the $R_\infty$ case). Then removing no stones at all from $R_\infty$ would be considered a simple partition. It has 0 altered rows and 0 height, hence corresponds to the term +1. 
\end{remark}

\section{$F_0$ Quiver}
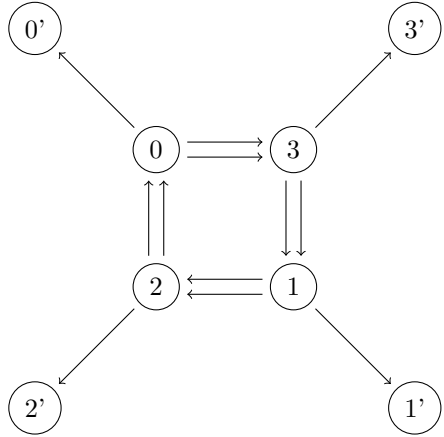
\begin{figure}[!htb]
    \centering
    \begin{tikzpicture}[round/.style={circle, draw, outer sep= 3},]
        \node[round] (0) [] {0};
        \node[round] (3) [right=of 0] {3};
        \node[round] (2) [below = of 0] {2};
        \node[round] (1) [below = of 3] {1};
        
        \node[round] (0') [above left=of 0] {0'};
        \node[round] (1') [below right=of 1] {1'};
        \node[round] (2') [below left=of 2] {2'};
        \node[round] (3') [above right=of 3] {3'};

        \draw[transform canvas={xshift=-1mm},->] (2) to (0);
        \draw[transform canvas={xshift=1mm},->] (2) to (0);
        
        \draw[transform canvas={yshift=-1mm},->] (1) to (2);
        \draw[transform canvas={yshift=1mm},->] (1) to (2);
        
        \draw[transform canvas={xshift=-1mm},->] (3) to (1);
        \draw[transform canvas={xshift=1mm},->] (3) to (1);
        
        \draw[transform canvas={yshift=-1mm},->] (0) to (3);
        \draw[transform canvas={yshift=1mm},->] (0) to (3);
        
        \draw[->] (0) to (0');
        \draw[->] (1) to (1');
        \draw[->] (2) to (2');
        \draw[->] (3) to (3');
    \end{tikzpicture}
    \caption{Framed $F_0$ quiver. $\mu = 01230123 \ldots$}
    \label{fig:F0}
\end{figure}

Next, we investigate the $F_0$ quiver, whose framed quiver is shown in Figure \ref{fig:F0}. We consider this example with respect to the mutation sequence $\mu = (0,1,2,3,0,1,2,3,\ldots)$. 

Here the cluster variables appear to converge, but this time in two phases. That is, the even-indexed cluster variables appear to converge to one limit, and the odd-indexed cluster variables appear to converge to another limit. Currently, we have only completed an analysis of the even-indexed cluster variables, whose limit generalizes the conifold case further. An explanation for the odd-indexed cluster variables, or the existence of this double sequence, has not yet been investigated. 

\par A table of the odd-indexed cluster variables. (Entries in the last two rows are truncated).

\begin{center}
\begin{tabular}{|c|r|r|}
\hline
$k$ & $F_k$ & $\tilde{F}_k$\\
\hline
1 & $y_0 + 1$ & $y_0 + 1$\\

3 & $y_0^2y_1^2y_2 + 2y_0^2y_1y_2 + y_0^2y_2 + y_0^2 + 2y_0 + 1$ & $y_0^2y_2^4 + y_1^2y_2 + 2y_0y_2^2 + 2y_1y_2 + y_2 + 1$\\

5 & $\ldots + 4y_0^2y_1y_2 + y_0^3 + 2y_0^2y_2 +
3y_0^2 + 3y_0 + 1$ & $\ldots + 4y_0y_1y_3^2 + y_0y_3^2 + 2y_0^2y_2 + 2y_0y_3 + y_0 + 1$\\

7 & $\ldots +
6y_0^2y_1y_2 + 4y_0^3 + 3y_0^2y_2 + 6y_0^2 + 4y_0 + 1$ & $\ldots 
+ 4y_1^2y_2y_3 + y_1^2y_2 + 2y_0y_2^2 + 2y_1y_2 + y_2 + 1$\\
\hline
\end{tabular}\\
\end{center}
\vspace{4.5mm}

\par A table of the even-indexed cluster variables. (Entries in the last two rows are truncated)
\begin{center}
\begin{tabular}{|c|r|r|}
\hline
$k$ & $F_k$ & $\tilde{F}_k$\\
\hline
2 & $y_1 + 1$ & $y_1 + 1$\\

4 & $y_0^2y_1^2y_3 + 2y_0y_1^2y_3 + y_1^2y_3 + y_1^2 + 2y_1 + 1$ & $y_0^2y_2^4y_3 + y_1^2y_3^4 + 2y_0y_2^2y_3 + 2y_1y_3^2 + y_3 + 1$\\

6 & $\ldots + 4y_0y_1^2y_3 + y_1^3 + 2y_1^2y_3 +
3y_1^2 + 3y_1 + 1$ & $\ldots + 4y_0^3y_1y_2^2 +
3y_1^3y_3^2 + 2y_0^2y_1y_2 + 2y_1^2y_3 + y_1 + 1 $\\

8 & $\ldots +
6y_0y_1^2y_3 + 4y_1^3 + 3y_1^2y_3 + 6y_1^2 + 4y_1 + 1$ & $\ldots + 4y_0^2y_2^3y_3 + 3y_1^2y_3^3 + 2y_0y_2^2y_3 + 2y_1y_3^2 + y_3 + 1 $\\
\hline
\end{tabular}\\
\end{center}
\vspace{4.5mm}

For the remainder of the discussion, we consider only the even-indexed cluster variables, and we re-index them from $F_2, F_4, F_6, \ldots$ to $F_1, F_2, F_3, \ldots$ in order to simplify notation.

\par Here is a larger number of stable terms. By identifying pairs of variables $y_0 = y_1$, $y_2 = y_3$, these terms collapse down to the conifold case. 

\begin{multline*}
\ldots + 6y_1^6y_3^5 + 4y_0^2y_1^5y_2y_3^3 + 10y_0^6y_1y_2^4 + 8y_0^4 y_1^2y_2^3y_3+ 8y_0^5y_1y_2^4 + 5y_1^5y_3^4\\ + 2y_0^2y_1^4y_2y_3^2 + 4y_0^5y_1y_2^3 + 4y_0^3y_1^2y_2^2y_3 + 6y_0^4y_1y_2^3 + 4y_1^4y_3^3 + y_0^4y_1y_2^2\\ + 4y_0^3y_1y_2^2 + 3y_1^3y_3^2 + 2y_0^2y_1y_2 + 2y_1^2y_3 + y_1 + 1
\end{multline*}

\subsection{4-color Aztec Diamond Pyramids}
The F-polynomials are, once again, partition functions of pyramids. 

\begin{definition}
Let $AD^{(4)}_k$ be the following 4-color Aztec diamond pyramid with $k$ white stones on the top layer. The next three layers down consist of black stones, yellow stones, and blue stones, respectively. 
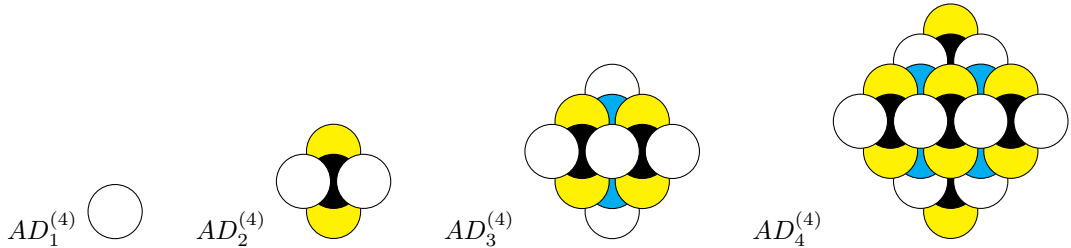
\begin{center}
$AD^{(4)}_1$ \begin{tikzpicture}[scale = 0.8]
        \foreach \x in {}
        \draw[fill = black] (\x, 0) circle (4.5mm);
        \foreach \x in {0}
        \draw[fill = white] (\x, 0) circle (4.5mm);
\end{tikzpicture}
\hspace{4.5mm}
$AD^{(4)}_2$ \begin{tikzpicture}[scale = 0.8]
        \foreach \x in {0} \foreach \y in {-0.5, 0.5}
        \draw[fill = yellow] (\x, \y) circle (4.5mm);
        \foreach \x in {0} \foreach \y in {0}
        \draw[fill = black] (\x, \y) circle (4.5mm);
        \foreach \x in {-0.5,0.5} \foreach \y in {0}
        \draw[fill = white] (\x, \y) circle (4.5mm);
\end{tikzpicture}
\hspace{4.5mm}
$AD^{(4)}_3$ \begin{tikzpicture}[scale = 0.8]
        \foreach \x in {0} \foreach \y in {-1, 0, 1}
        \draw[fill = white] (\x, \y) circle (4.5mm);
        
        \foreach \x in {0} \foreach \y in {-0.5, 0.5}
        \draw[fill = cyan] (\x, \y) circle (4.5mm);
        
        \foreach \x in {-0.5,0.5} \foreach \y in {-0.5, 0.5}
        \draw[fill = yellow] (\x, \y) circle (4.5mm);
        
        \foreach \x in {-0.5, 0.5} \foreach \y in {0}
        \draw[fill = black] (\x, \y) circle (4.5mm);
        \foreach \x in {-1,0,1} \foreach \y in {0}
        \draw[fill = white] (\x, \y) circle (4.5mm);
\end{tikzpicture}
\hspace{4.5mm}
$AD^{(4)}_4$ \begin{tikzpicture}[scale = 0.8]
        \foreach \x in {0} \foreach \y in {-1.5, -0.5, 0.5, 1.5}
        \draw[fill = yellow] (\x, \y) circle (4.5mm);
        
        \foreach \x in {0} \foreach \y in {-1, 0, 1}
        \draw[fill = black] (\x, \y) circle (4.5mm);
        
        \foreach \x in {-0.5,0.5} \foreach \y in {-1, 0, 1}
        \draw[fill = white] (\x, \y) circle (4.5mm);
        
        \foreach \x in {-0.5, 0.5} \foreach \y in {-0.5, 0.5}
        \draw[fill = cyan] (\x, \y) circle (4.5mm);
        
        \foreach \x in {-1,0,1} \foreach \y in {-0.5, 0.5}
        \draw[fill = yellow] (\x, \y) circle (4.5mm);
        
        \foreach \x in {-1, 0, 1} \foreach \y in {0}
        \draw[fill = black] (\x, \y) circle (4.5mm);
        
        \foreach \x in {-1.5, -0.5, 0.5, 1.5} \foreach \y in {0}
        \draw[fill = white] (\x, \y) circle (4.5mm);
    \end{tikzpicture}

\end{center}

\end{definition}

We carry over the definition of a partition of $AD^{(4)}_k$ from previous sections unchanged. Now, for any partition $P$ of $AD^{(4)}_k$, its \textbf{weight} is
$$weight(P)= y_0^{\text{ \# yellow removed}}y_1^{\text{\# white removed}}y_2^{\text{ \# blue removed}}y_3^{\text{\# black removed}}$$

\begin{theorem}[Elkies-Kuperberg-Larsen-Propp, 1992]
The F-polynomials are partition functions of $AD^{(4)}_k$.
$$F_k = \sum_{\text{Partitions $P$ of } AD^{(4)}_{k}} weight(P)$$
\end{theorem}

\subsection{Stabilization and Interpretation of the Limit}
\begin{theorem}
For the $F_0$ quiver, $\lim_{k\to \infty} \tilde{F}_k$ converges as a formal power series. 
\end{theorem}
\begin{proof}
Details are omitted, but the proof is essentially identical to that of Theorem 4.6, except that there are four variables or four colors to keep track of instead of two. Note that by identifying blue stones with black stones, and identifying yellow stones with white stones, we recover the two-color Aztec diamonds from the previous section. This corresponds to identifying pairs of variables $y_0 = y_1$, $y_2 = y_3$, which collapses polynomials down to those from the previous section. The $F_0$ quiver can hence be seen as an "unfolded" version of the conifold. 
\end{proof}

Analogous to the previous section, there is a combinatorial interpretation of the limit in terms of$AD_\infty^{(4)}$, the infinite 4-color Aztec diamond, which generalizes the expression from the conifold case. 

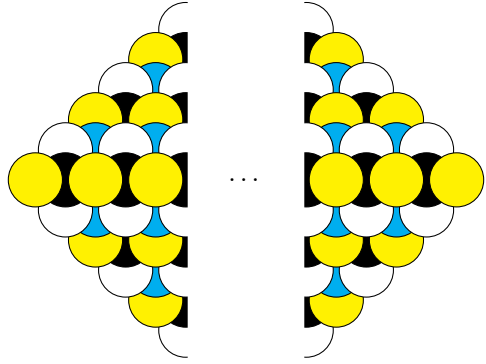
\begin{center}
\begin{tikzpicture}[scale = 0.8]
                
        \foreach \x in {-1,1}
            \foreach \y in {2.5, 1.5, 0.5, -0.5, -1.5, -2.5}
                \draw[fill = white] (\x, \y) circle (4.5mm);
                
        \foreach \x in {-1,1}
            \foreach \y in {2, 1, 0, -1,-2}
                \draw[fill = black] (\x, \y) circle (4.5mm);
                
        \foreach \x in {-1.5, 1.5}
            \foreach \y in {2, 1, 0, -1, -2}
                \draw[fill = yellow] (\x, \y) circle (4.5mm);

        \foreach \x in {-1.5, 1.5}
            \foreach \y in {-1.5, -0.5, 0.5, 1.5}
                \draw[fill = cyan] (\x, \y) circle (4.5mm);
                
        \foreach \x in {-2,-1,1,2}
            \foreach \y in {-1.5, -0.5, 0.5, 1.5}
                \draw[fill = white] (\x, \y) circle (4.5mm);
                
        \foreach \x in {-2, -1, 1, 2}
            \foreach \y in {-1, 0, 1}
                \draw[fill = black] (\x, \y) circle (4.5mm);
                
        \foreach \x in {-2.5,-1.5,1.5,2.5}
            \foreach \y in {-1, 0, 1}
                \draw[fill = yellow] (\x, \y) circle (4.5mm);
        
        \foreach \x in {-2.5, -1.5, 1.5, 2.5}
            \foreach \y in {-0.5, 0.5}
                \draw[fill = cyan] (\x, \y) circle (4.5mm);
                
        \foreach \x in {-3, -2,-1,1,2, 3}
            \foreach \y in {-0.5, 0.5}
                \draw[fill = white] (\x, \y) circle (4.5mm);
    
        \foreach \x in {-3, -2, -1, 1, 2, 3}
            \foreach \y in {0}
                \draw[fill = black] (\x, \y) circle (4.5mm);
                
        \foreach \x in {-3.5, -2.5,-1.5,1.5,2.5, 3.5}
            \foreach \y in {0}
                \draw[fill = yellow] (\x, \y) circle (4.5mm); 
                
        \draw[color = white, fill = white] (-.96,-3) rectangle (.96, 3);
        \node () at (0,0) {\ldots};
                
\end{tikzpicture}
\end{center}

\begin{definitions}
Let simple partitions be defined as before, and height $h(r)$ of a row $r$ be defined as before. Note that each row in the pyramid decomposition is either a black/white row or a blue/yellow row. For any simple partition $P$, let $R_1(P)$ be the set of altered black/white rows and $R_2(P)$ be the set of altered blue/yellow rows. Define \textbf{heights}:
$$h_1 = \sum_{r \in R_1} h(r)$$
$$h_2 = \sum_{r \in R_2} h(r)$$
Then let:
$$x_1 = \sum_{r \in R_1} (\# \text{non-removed white stones in } r)$$
$$= \sum_{r \in R_1} (\# \text{non-removed black stones in } r)$$
and:
$$x_2 = \sum_{r \in R_2} (\# \text{non-removed yellow stones in } r)$$
$$= \sum_{r \in R_2} (\# \text{non-removed blue stones in } r)$$
\end{definitions}

\begin{proposition}
For the $F_0$ quiver $$\lim_{k \to \infty} \tilde{F}_k = \sum_{\text{simple partitions of } AD_\infty^{(4)}} y_0^{x_2 + h_2 + \#R_2}y_1^{x_1 + h_1 + \#R_1}y_2^{x_2 + h_2}y_3^{x_1 + h_1}$$
\label{prop:T}
\end{proposition}

\begin{proof}
Details omitted, but essentially similar to the analogous proposition in the previous section.
\end{proof}

\section{Conclusion}

\subsection{Open Questions}
Numerous questions remain to be answered, including: 
\begin{enumerate}
\item How do we explain the behavior seen in the $F_0$ quiver, where the transformed cluster variables split into two distinct sequences? Can we predict for which quivers such a fork occurs?
\item How can we prove that the odd-indexed F-polynomials for the $F_0$ quiver also stabilize? What is a combinatorial interpretation for these functions? 
\item Each of the three examples presented throughout this paper generalizes the previous in a natural way. What family of quivers and mutation sequences does this generalization ultimately extend to? 
\item Eager and Franco originally observed apparent stabilization for the dP1 quiver. So far, we have not investigated this case.
\item What characterizes the class of quivers and mutation sequences for which stabilization occurs? What is the underlying explanation that causes stabilization? And what significance does this have in the context of quiver gauge theories? 
\end{enumerate}

\subsection{Acknowledgments}  

 This research was carried out as part of the 2016 Summer REU program at the School of Mathematics, University of Minnesota, Twin Cities, and was supported by NSF RTG grant DMS-1148634. Special thanks my mentor, Gregg Musiker, and my TA, Ben Strasser, for this project. 

\bibliographystyle{plain}
\bibliography{bibliography}

\end{document}